\documentclass[a4paper, 11pt]{amsart}
\usepackage{fullpage}
\title
{Decomposition theorem for good moduli morphisms}
\date{}
\author{Tasuki Kinjo}

\makeatletter
 
  \@addtoreset{equation}{section}
 \makeatother

\usepackage{amsmath, amssymb, amsthm, amscd, comment, mathtools, amsfonts, color}
\usepackage{stmaryrd}
\usepackage[frame,cmtip,curve,arrow,matrix,line,graph]{xy}
\usepackage[mathscr]{euscript}
\usepackage{mathrsfs}
\usepackage{tikz}
\usetikzlibrary{intersections, calc}

\usepackage[T1]{fontenc}
\usepackage{hyperref}

\theoremstyle{plain}

\newtheorem{thm}{Theorem}[section]
\newtheorem{prop}[thm]{Proposition}
\newtheorem{def-prop}[thm]{Definition-Proposition}

\newtheorem{lem}[thm]{Lemma}
\newtheorem{cor}[thm]{Corollary}
\newtheorem*{thm*}{Theorem}

\theoremstyle{definition}
\newtheorem{defin}[thm]{Definition}

\newtheorem*{ACK}{Acknowledgement}

\theoremstyle{remark}
\newtheorem{rmk}[thm]{Remark}
\newtheorem*{rmk*}{Remark}

\usepackage{enumitem}
\newlist{thmlist}{enumerate}{1}
\setlist[thmlist]{label=(\roman{thmlisti}), ref=\thethm(\roman{thmlisti})}

\newcommand{\bA}{\mathbb{A}}

\newcommand{\bC}{\mathbb{C}}

\newcommand{\bG}{\mathbb{G}}

\newcommand{\bQ}{\mathbb{Q}}

\DeclareMathOperator{\Hom}{Hom}

\DeclareMathOperator{\codim}{codim}
\DeclareMathOperator{\GL}{GL}

\DeclareMathOperator{\pr}{pr}

\DeclareMathOperator{\fib}{fib}

\newcommand{\cE}{{\mathcal E}}
\newcommand{\cF}{{\mathcal F}}

\newcommand{\cH}{{\mathcal H}}

\newcommand{\cU}{{\mathcal U}}
\newcommand{\cV}{{\mathcal V}}

\newcommand{\fM}{{\mathfrak M}}

\newcommand{\fS}{{\mathfrak S}}
\newcommand{\fT}{{\mathfrak T}}

\newcommand{\fX}{{\mathfrak X}}
\newcommand{\fY}{{\mathfrak Y}}
\newcommand{\fZ}{{\mathfrak Z}}

\newcommand{\GIT}{{/\! \! /}}

\newcommand\freefootnote[1]{%
  \let\thefootnote\relax%
  \footnotetext{#1}%
  \let\thefootnote\svthefootnote%
}

\renewcommand{\det}{\mathrm{det}}

\renewcommand{\H}{\mathrm{H}}

\begin{document}

\begin{abstract}

In this short note, we will explain that the good moduli space morphisms behave as if they are proper when we consider sheaf operations, though they are not separated.
For example, the decomposition theorem and the base change theorem hold for these morphisms, which have applications to the cohomological study of moduli spaces.

\end{abstract}

\maketitle

\section{Main results}

Throughout the paper,
we will work over the complex number field.

\subsection*{Base change theorem}

For a stack $\fX$ of finite type,
we let $D^+_c(\fX)$  denote the lower bounded constructible derived category of sheaves of $\bQ$-vector spaces with respect to the analytic topology over $\fX$.
Let $f \colon \fX \to \fY$ be a morphism.
We say that the
\textit{base change theorem} holds for $f$ if the following is true for any $\cF \in D^+_c(\fX)$:

\begin{itemize}
    \item For a finite type morphism $t \colon \fT \to \fY$, 
    the Beck--Chevalley map
     \[
        t^* f_*  \cF \to f'_* t'{}^{ *} \cF
        \]
    is invertible, where $f' \colon \fT \times_{\fY} \fX \to \fT$ and   
    $t' \colon \fT \times_{\fY} \fX \to \fX$ are base change of $f$ and $t$ respectively.

\end{itemize}

\subsection*{Weight preservation}

For a stack $\fX$ of finite type, we let $D^+_H(\fX)$ denote the lower bounded derived category of mixed Hodge modules on $\fX$,
which was constructed very recently by Swann Tubach \cite{Tub2}
as a part of his extension of the six-functor formalism of mixed Hodge modules to Artin stacks.
We say that $f$ \textit{preserves weights} if the following condition holds:
\begin{itemize}
    \item For a an object $M \in D^+_H(\fX)$ which is pure of weight $n$,
    $f_* M$ is also pure of weight $n$.
\end{itemize}

\subsection*{Main theorem}

The aim of this note is to prove the following statement:

\begin{thm}\label{thm:main}
    Let $\fX$ be a finite type Artin stack with affine diagonal admitting a good moduli space $h \colon \fX \to X$.
  Then $h$ satisfies the base change theorem and preserves weights.
\end{thm}

When $\fX$ is the moduli space of quiver representations, this theorem was proved by  Davison and Meinhardt in \cite[\S 4.1]{dm20}  using the approximation of the morphism $h$ by the forgetful map from the moduli space of framed quiver representations which is proper and representable.

\subsection*{Strategy of the proof}

Our proof is also based on an approximation, but we need some new idea since a naive generalization of the moduli space of framed quiver representation for a general affine GIT quotient might contain strictly semistable points.
We will overcome this difficulty by combining induction on the dimension of stabilizer groups of $\fX$ and Luna's \'etale slice theorem \cite[Theorem 4.12]{ahr20}.


\section{Approximately proper morphisms}

We will introduce the notion of \textit{approximately proper morphisms},
which behave as proper maps when we consider sheaf-operations.

\begin{defin}\label{defin:AP}
    Define a class $\mathrm{AP}$ of morphisms between finite type Artin stacks to be the minimal one satisfying the following conditions:
    \begin{enumerate}[label={\rm(AP\arabic*)},ref={\upshape(AP\arabic*)}]

    \item \label{item:proper}  All proper morphisms represented by Deligne--Mumford stacks are in $\mathrm{AP}$. 
    
    \item \label{item:comp} $\mathrm{AP}$ is closed under composition.

    \item \label{item:etale} Being in $\mathrm{AP}$ can be checked \'etale locally on the target.

    \item \label{item:proj} Let $f \colon \fX \to \fY$ be a morphism and $\fZ$ be a non-empty finite type Artin stack.
    Assume that the composition
    \[
    \fX \times \fZ \xrightarrow[]{\pr_1} \fX \xrightarrow[]{f} \fY
    \]
    is in $\mathrm{AP}$.
    Then $f$ is in $\mathrm{AP}$.
    
    \item \label{item:approximate} Let $f \colon \fX \to \fY$ be a morphism. Assume that for each integer $n$, there exists a vector bundle $\cE_n$ over $\fX$ and and an open subset $\cU_n \subset \cE_n$ with $\codim_{\cE_n |_{x}}((\cE_n \setminus \cU_n) |_{x}) > n$ for each $x \in \fX$ such that the composition 
    \[
    \cU_n \hookrightarrow \cE_n \to \fX \xrightarrow[]{f} \fY
    \]
    is in $\mathrm{AP}$.
    Then $f$ is in $\mathrm{AP}$.
    \end{enumerate}

    A morphism contained in $\mathrm{AP}$ is called approximately proper.
\end{defin}

\begin{prop}\label{pro:AP_main}
   An approximately proper morphism satisfies the base change theorem and preserves weights. 
\end{prop}

\begin{proof}
    Let $\mathrm{BW}$ be the class of morphisms between finite type stacks satisfying the base change theorem and preserving weights.
    It is enough to show the inclusion $\mathrm{AP} \subset \mathrm{BW}$.
    Equivalently, it is enough to show that the claim \ref{item:proper} -- \ref{item:approximate} in Definition \ref{defin:AP} holds 
    after replacing
    $\mathrm{AP}$ with $\mathrm{BW}$.
    The statement \ref{item:proper} follows from \cite[Thereom A.0.8]{ACV03} and \cite[Theorem 3.10, Proposition 3.21]{Tub2}.
    The statement \ref{item:comp} and  \ref{item:etale} are obvious.
    We now prove \ref{item:proj}.
    Take $\cF \in D^+_c(\fX)$.
    Then we have an isomorphism
    \[
    \pr_{1, *} \pr_1^* \cF \cong \cF \otimes \H^{\bullet}(\fZ).
    \]
    Therefore the base change theorem for $f \circ \pr_1$ applied to the complex $\pr_1^* \cF$ implies the base change theorem for $f$ and the complex $\cF$.
    The weight-preservation can be proved analogously.
    Finally we prove \ref{item:approximate}.
    Take a complex $\cF \in D^{\geq 0}_c(\fX)$.
    Let $p_n \colon \cU_n \to \fX$ be the natural morphism and $t \colon \fT \to \fY$ be a finite type morphism.
    Consider the following diagram:
    \[
    \xymatrix{
    {\cV_n} \ar[r]^-{p_n'} \ar[d]^-{t''}
    & {\fS}  \ar[r]^-{f'} \ar[d]^-{t'}
    & {\fT} \ar[d]^-{t} \\
    {\cU_n} \ar[r]^-{p_n}
    & {\fX} \ar[r]^-{f}
    & {\fY}
    }
    \]
    where two squares are Cartesian.
    Consider the following commutative diagram:
    \[
    \xymatrix{
    {t^* f_* \cF}
    \ar[rr]
    \ar[d]
    & {}
    & {f'_* t'^* \cF} \ar[d] \\
    {t^* f_* p_{n, *} p_n^* \cF}
    \ar[r]^-{\simeq}
    & { f'_* p'_{n, *} t''^*  p_n^* }
    \ar[r]^-{\simeq}
    & {f'_* p'_{n, *} p_n'^* t'^* \cF .}
    }
    \]
    Here, the left bottom horizontal map is invertible by the assumption that $f \circ p_n$ is in $\mathrm{BW}$.
    By the assumption on the codimension of $\cE_n \setminus \cU_n$, 
    we have
    \[
    \fib(\cF \to p_{n, *} p_{n}^* \cF) \in D_c^{> 2n}(\fX),
    \quad \fib(t'^* \cF \to p_{n, *}' p_{n}'^* t'^* \cF) \in D_c^{> 2n}(\fX).
    \]
    In particular, we have 
    \[
    \fib(t^* f_* \cF \to f'_* t'^* \cF) \in D_c^{> 2n}(\fX)
    \]
    for any $n$, hence the Beck--Chevalley map $t^* f_* \cF \to f'_* t'^* \cF$ is invertible.
    The weight-preservation can be proved analogously.
    
\end{proof}

\begin{prop}\label{prop:good_AP}
    Let $\fX$ be an Artin stack with affine  diagonal admitting a good moduli space $h \colon \fX \to X$.
    Then $h$ is approximately proper.
\end{prop}

This proposition will be proved in Section \ref{sec:good_AP}.

\begin{proof}[Proof of Theorem \ref{thm:main}]
    It is a direct consequence of Proposition \ref{pro:AP_main} and Proposition \ref{prop:good_AP}.
\end{proof}

\section{Key technical lemma}

The proof of Proposition \ref{prop:good_AP} will be based on the induction on the dimension of the stabilizer groups.
The following technical lemma will be used to reduce the dimension of the stabilizers:

\begin{lem}\label{lem:key}
Let $G$ be a reductive group and $n$ be a positive integer. Then there exists a $G$-representation $(V_n, \mu_n')$ with the following property:
If we define an action $\mu_n$ of $G \times \bG_m$ on $V_n$ by $\mu_n' \cdot \chi $ where $\chi$ denotes the second projection,
the following conditions hold:
\begin{itemize}
    \item Let $V_{n}^{\mathrm{ss}} \subset V_n$ be the semistable locus of the action $\mu_n$ with respect to the linearization $\chi$.
    Then $V_{n} \setminus V_n ^{\mathrm{ss}}$ has codimension greater than $n$.
    
    \item For an affine $G$-variety $(Y, \nu)$, consider $G \times \bG_m$-action on $Y \times V_n$ given by $(\nu \circ \pr_1, \mu_n)$ with the linearization $\chi$.
    Then for any semistable point $(y, v) \in (Y \times V)^{\mathrm{ss}}$, the stabilizer group
    $(G \times \bG_m)_{(y, v)} \subset G \times \bG_m$ has codimension greater than or equal to $2$.
    
\end{itemize}
\end{lem}

\begin{proof}
    We fix a maximal torus $T \subset G$ and $G^{\circ}$ be the connected component of $G$ containing the unit.
Then there exists a $G^{\circ}$-representation $G^{\circ} \hookrightarrow \GL(W')$ such that the set of $T$-weights generates the weight space $\Hom(T, \bG_m)_{\bQ}$ and $W'$ does not contain the trivial representation as a direct summand.
We define a $G$-representation $W$ as the induced representation from $W'$.
We set $V_n$ to be the direct sum of $m_n$ copies of $V = W \oplus W^*$ for large enough integer $m_n$. We claim that $V_n$ satisfies the desired property.

We check the first condition. Consider the following map of algebraic groups
\[
\tau \colon G \times \bG_m \to \GL(V) \times \bG_m \to \GL(V)
\]
where the latter map is given by the multiplication.
It follows from the assumption that for any $1$-PS $\lambda \colon \bG_m \to G \times \bG_m$,
the composition $\tau \circ \lambda$ is non-trivial.
Therefore $\tau$ has a finite kernel.
Also, it implies that the unstable locus of $V_n$ as a $G \times \bG_m$-representation with respect to the linearization $\chi$ is contained in the unstable locus as a $\GL(V)$-representation with respect to the linearization $\det$.
Since the unstable locus of $V_n = V^{m_n}$ as a $\GL(V)$-representation
consists of tuples $(v_1, \ldots, v_{m_n}) \in V^{m_n}$ which does not span $V$,
the complement $V_n \setminus V_n^{\mathrm{ss}}$
has its codimension larger than $n$ if $m_n$ is large enough.
In particular, the first condition is satisfied.

Now we check the second condition.
Take $(y, v) \in  Y \times V_{n}$ and assume that the stabilizer group $(G \times \bG_m)_{(y, v)}$ has codimension at most one.
If $v$ is the origin, 
then the $1$-PS 
 \[
 \{1\} \times \bG_m \xrightarrow[]{ t \to t^{-1}} \{1\} \times \bG_m \hookrightarrow G \times \bG_m
 \]
 destabilizes $(y, v)$.
 Assume now that $v$ is not the origin.
Since the orbit $(G^{\circ} \times \bG_m) \cdot (y, v)$ has dimension one, we have the equality
\[
(G^{\circ} \times \bG_m) \cdot (y, v) = \{ y \} \times (\bG_m \cdot v).
\]
We claim the unstability of $(y, v)$.
Let $\mu \colon G^{\circ} \to \bG_m$ be the character corresponding to the $G$-action on $\bC \cdot v$.
Note that $\mu \neq 0$, since $V_n$ does not contain the trivial $G^{\circ}$-representation.
Therefore we can take a $1$-PS $\lambda \colon \bG_m \to G^{\circ}$ such that $\mu \circ \lambda$ has a positive weight.
We take a $1$-PS $\tilde{\lambda} = (\lambda^l, z^{-1}) \colon \bG_m \to G \times \bG_m$.
If $l\geq 1$,
it is clear that $\tilde{\lambda}$ destabilizes the point $(y, v)$.
\end{proof}

\section{Proof of Proposition \ref{prop:good_AP}}\label{sec:good_AP}

We will prove Proposition \ref{prop:good_AP} by the induction on the maximal dimension of the stabilizer groups.

Assume that the maximal dimension of the stabilizer of $\fX$ is zero.
By using \cite[Theorem 4.12]{ahr20} and \ref{item:etale}, we may assume $\fX = [Y / G]$ where $\fX$ is an affine variety and $G$ is a finite group.
In this case, the map $h \colon \fX \to X$ is proper and represented by Deligne--Mumford stacks, hence the claim follows from \ref{item:proper}.

Assume now that we have proved the statement for stacks whose maximal stabilizer dimension is less than $d$ and the maximal stabilizer dimension of $\fX$ is $d$.
By using \cite[Theorem 4.12]{ahr20} and \ref{item:etale}, we may assume $\fX = [Y / G]$ where $\fX$ is an affine variety and $G$ is a reductive group of dimension $d$.
To prove that the map $h \colon \fX \to X$ is approximately proper, by \ref{item:proj}, it is enough to show the composite $h' \colon \fX \times B \bG_m \to \fX \to X$ is approximately proper.

For an integer $n$, take a $G \times \bG_m$-representation $V_n$ in Lemma \ref{lem:key}.
Consider the following diagram:
\[
    \xymatrix@C=80pt{
{[(Y \times V_n)^{\mathrm{ss}} / G \times  \bG_m]}
\ar[d]^-{h_n}
\ar[r]^-{p_n}   
& {[Y / G] \times B \bG_m}
\ar[d]^-{h'} \\
{(Y \times V_n)^{\mathrm{ss}} \GIT G \times  \bG_m}
\ar[r]^-{\bar{p}_n} 
& {Y \GIT G.}
    }
\]
Note that $\bar{p}_n$ is proper since it is given by the variation of GIT.
Also, by the second property of $V_n$ in Lemma \ref{lem:key} and the induction hypothesis, we see that $h_n$ is approximately proper.
Therefore the map $h' \circ p_n =  \bar{p}_n \circ h_n$ is approximately proper by  \ref{item:comp}.
Since we have the inclusion $Y \times V_n^{\mathrm{ss}} \subset (Y \times V_n)^{\mathrm{ss}}$, using the first property of $V_n$ in Lemma \ref{lem:key},
we conclude that $h'$  is approximately proper by \ref{item:approximate}.

\section{Applications}

\subsection{Purity of moduli spaces}

The following statement is an immediate consequence of Theorem \ref{thm:main}:

\begin{thm}\label{thm:pure}
    Let $\fX$ be a smooth Artin stack with affine diagonal admitting a good moduli space $h \colon \fX \to X$.
    Then the mixed Hodge complex $h_* \bQ_{\fX}$ is pure.
    In particular, there exists an isomorphism
    \[
    h_* \bQ_{\fX} \cong \bigoplus_{i} {}^p \cH^i(h_* \bQ_{\fX})[-i].
    \]
\end{thm}

Theorem \ref{thm:pure} implies the following statement:

\begin{cor}\label{cor:pure}
    We let $\fX$ be one of the following stacks:
    \begin{enumerate}[label=\upshape{(\Alph*)}]
        \item  \label{item:delpezzo} Let $S$ be a smooth del pezzo surface, $H$  an ample divisor and take $\alpha \in \H^{\bullet}(S)$.
        Set $\fX = \fM_{S, \alpha}^{H-\mathrm{ss}}$ the moduli stack of $H$-semistable sheaves on $S$ with Chern character $\alpha$.

        \item \label{item:Bundle} Let $C$ be a smooth projective curve and $G$ be a reductive group. Set $\fX = \mathcal{B}\mathrm{un}^{\mathrm{ss}}_G(C)$ the moduli stack of semistable $G$-bundles on $C$.

        \item \label{item:Higgs} Let $C$ and $G$ be as in \ref{item:Bundle} and $L$ be a line bundle with 
        $\H^{0} (C, L \otimes \omega_C^{-1}) > 0$.
        Set $\fX = \mathcal{H}\mathrm{iggs}^{L, \mathrm{ss}}_G(C)$ the moduli stack of semistable $L$-twisted $G$-Higgs bundles on $C$.
    \end{enumerate}

    Then the Borel--Moore homology $\H_{\bullet}^{\mathrm{BM}}(\fX)$ is pure.
\end{cor}






\begin{proof}
    We first deal with the stacks \ref{item:delpezzo}.
    If $\alpha$ corresponds to non-zero dimensional sheaves, 
    it easily follows that the stack $\fM_{S, \alpha}^{H-\mathrm{ss}}$ is smooth, hence the claim follows by the properness of the good moduli space and Theorem \ref{thm:pure}.
    If $\alpha$ corresponds to zero-dimensional sheaves,
    the claim follows from \cite[Theorem 7.1.6]{kv19} together with the fact that the cohomological Hall algebra preserves the mixed Hodge structures.
    
    The statement for the stacks \ref{item:Bundle} is obvious from Theorem \ref{thm:pure}, the smoothness of $\mathcal{B}\mathrm{un}^{\mathrm{ss}}_G(C)$ and the properness of the  good moduli spaces.
    
    Now we prove the statement for stacks \ref{item:Higgs}.
    Firstly it follows from \cite[Proposition 5.5]{her23} that $\mathcal{H}\mathrm{iggs}^{L, \mathrm{ss}}_G(C)$ is smooth.
    Consider the following composite
    \[
        \mathcal{H}\mathrm{iggs}_{G}^{L, \mathrm{ss}}(C) \xrightarrow[]{h} \mathrm{Higgs}_{G}^{L, \mathrm{ss}}(C) \xrightarrow[]{\mathrm{Hit}} B
    \]
    where the former map is the good moduli space morphism and $\mathrm{Hit}$ is the Hitchin fibration.
    By Theorem \ref{thm:pure}, the complex $h_* \bQ_{\mathcal{H}\mathrm{iggs}_{G}^{L, \mathrm{ss}}(C)}$ is pure.
    Further, since the map $\mathrm{Hit}$ is proper by \cite[Thm.II.4]{fal93}, the complex $(\mathrm{Hit} \circ h)_* \bQ_{\mathcal{H}\mathrm{iggs}_{G}^{L, \mathrm{ss}}(C)}$ is pure.
    Consider the contracting $\bG_m$-action on $B$, whose good moduli space is the point.
    Applying Theorem \ref{thm:main} for the $\bG_m$-action on $B$, we conclude that the cohomology 
    \[
    \H^{\bullet}\left({\mathcal{H}\mathrm{iggs}_{G}^{L, \mathrm{ss}}(C)}\right) \cong \H^{\bullet}\left(B, (\mathrm{Hit} \circ h)_* \bQ_{\mathcal{H}\mathrm{iggs}_{G}^{L, \mathrm{ss}}(C)}\right)
    \]
    is pure as desired.
\end{proof}

\begin{rmk}
    The purity for the stacks \ref{item:delpezzo} when $\alpha$ corresponds to a non-zero dimensional sheaves is already proved in \cite[Theorem 0.11]{KLMP}.
    Further, they prove the tautological generation of the cohomology ring, which is not available with our method.
\end{rmk}

\begin{rmk}
    The purity for the stacks \ref{item:Bundle} essentially goes back to  the work of Atiyah--Bott \cite{AB83} (at least when $G$ is connected).
    Namely, they proved that  the rational cohomology of the moduli stack $\mathcal{B}\mathrm{un}_G(C)$ of (not necessarily semistable) $G$-bundles over $C$ is freely generated by tautological classes and the natural map
    \[
    \H^{\bullet}(\mathcal{B}\mathrm{un}_G(C)) \to \H^{\bullet}(\mathcal{B}\mathrm{un}^{\mathrm{ss}}_G(C))
    \]
    is surjective.
\end{rmk}

\begin{rmk}
    One can weaken the assumption in \ref{item:Higgs} to $\deg L > 2 g(C) -2$ for the smoothness of $\mathcal{H}\mathrm{iggs}_{G}^{L, \mathrm{ss}}(C)$; this follows from a deformation theory argument,
    and the detail will be explained elsewhere.
    In particular, the same conclusion holds for this generality.
\end{rmk}

\begin{rmk}
The moduli stacks \ref{item:delpezzo}, \ref{item:Bundle} and \ref{item:Higgs} are naturally realized as the open strata of a $\Theta$-stratification of a stack $\widetilde{\fX}$. The localization sequence with respect to the $\Theta$-stratification together with Corollary \ref{cor:pure} implies
the purity of $\H^{\mathrm{BM}}_{\bullet}(\widetilde{\fX})$ and
that the Kirwan map
\[
\H^{\mathrm{BM}}_{\bullet}(\widetilde{\fX}) \to \H^{\mathrm{BM}}_{\bullet}(\fX)
\]
is surjective. See \cite[Corollary 4.1]{Hlp15} for a related discussion.
\end{rmk}

\subsection{Cohomological Ehresmann fibration theorem}

Theorem \ref{thm:main} can be applied to the study of the cohomology of smooth stacks in family.
The following statement is a direct consequence of Corollary \ref{cor:van_BC}. See \cite[Theorem 3.8, Corollary 3.9]{deINFH} for a related discussion, where they prove a similar statement for the (intersection) cohomology of the good moduli spaces.

\begin{thm}
    Let $S$ be a connected finite type scheme and $\fX \to S$ be a family of smooth stacks with affine diagonal. 
    Assume that $\fX$ admits a good moduli space $X$ which is proper over $S$.
    Then for any $s, t \in S$, there exists a (non-canonical) isomorphism between the cohomology of the fibres $\H^*(\fX_s)$ and $\H^*(\fX_t)$ preserving the ring structure.
    Further, their Hodge numbers coincide.
\end{thm}

\appendix

\section{Base change theorem for the vanishing cycle functors}

We will discuss the commutation between the vanishing cycle functor and the push-forward:

\begin{prop}
    Let $f \colon \fX \to \fY$ be a morphism of finite type stacks and  $g \colon \fY \to \bA^1$ be a regular function.
    Set $\fX_0 \coloneqq (g \circ f)^{-1}(0)$ and $\fY_0 \coloneqq g^{-1}(0)$ and
     $f_0 \colon \fX_0 \to \fY_0$ denote the restriction of $f$.
    Assume that the base change theorem holds for $f$ and $\cF \in D^+_c(\fX)$ be a constructible complex.
    Then the natural maps 
    \[
    \varphi_g f_* \cF \to f_{0, *} \varphi_{g \circ f} \cF, \quad
     \psi_g f_* \cF \to f_{0, *} \psi_{g \circ f} \cF
    \]
    are invertible.
    
\end{prop}

\begin{proof}
    It is enough to prove the claim for the nearby cycle functor.
    By the definition of the nearby cycle functor \cite[Definition 3.31]{Tub2},
    we may further reduce to the unipotent part of the nearby cycle functor $\psi^{\mathrm{uni}}$.
    Consider the following diagram:
    \[
    \xymatrix{
    {\fX_0}
    \ar[r]^-{i}
    \ar[d]^-{f_0}
    & {\fX} \ar[d]^-{f}
    & {\fX_{\neq 0}} \ar[l]_-{j} \ar[d]^-{f_{\neq 0}} \\
    {\fY_0}
    \ar[r]^-{i'}
    & {\fY}
    & {\fY_{\neq 0}} \ar[l]_-{j'}
    }
    \]
    where $j$ and $j'$ are complementary open inclusion.
    Recall that the complex $i^* j_* j^* \cF$ is equipped with the action of $\H^{\bullet}( \bG_m) \simeq \bQ[t^1]$ with $\deg t^1 = 1$ and we have a formula
    \[
    \psi^{\mathrm{uni}}(\cF) \simeq i^* j_* j^* \cF \otimes_{\bQ[t^1]} \bQ.
    \]
    In particular, it is enough to prove that the base change transform
    \[
    i'^* j'_* j'^* f_* \cF \to f_{0, *} i^* j_* j^* \cF
    \]
    is invertible. However this follows from the assumption that $f$ satisfies the base change theorem.
\end{proof}

The above proposition together with Theorem \ref{thm:main} implies the following corollary:

\begin{cor}\label{cor:van_BC}
    Let $\fX$ be an Artin stack with affine diagonal admitting a good moduli space $h \colon \fX \to X$.
    Let $g \colon \fX \to \bA^1$ be a regular function and $\bar{g} \colon X \to \bA^1$ be the induced morphism.
    Set $\fX_0 \coloneqq g^{-1}(0)$, $X_0 \coloneqq \bar{g}^{-1}(0)$ and $h_0 \colon \fX_0 \to X_0$ be the restriction of $h$.
    Then for any $\cF \in D^+_c(\fX)$, the natural maps
    \[
        \varphi_{\bar{g}} h_* \cF \to h_{0, *} \varphi_{g } \cF, \quad
         \psi_{\bar{g}} h_* \cF \to h_{0, *} \psi_{g} \cF
        \]
        are invertible.
\end{cor}

\begin{ACK}
    The author thanks Andr\'es Ib\'a\~nez N\'u\~nez and Swan Tubach for a helpful discussion related to this work.
    He also thanks Andres Fernandez Herrero for useful comments on the preliminary version of this paper.
    T.K. was supported by JSPS KAKENHI Grant Number 23K19007.
\end{ACK}

\bibliographystyle{amsalpha}
\bibliography{biblio}

\end{document}